\documentclass[12pt,letterpaper]{article}

\usepackage[utf8]{inputenc}
\usepackage{amsmath,amsfonts,amsthm,amssymb,verbatim,etoolbox,color,mathtools}
\usepackage{hyperref}
\usepackage{authblk}
\usepackage{mathrsfs}
\usepackage{fullpage}
\usepackage{dsfont}

\usepackage[maxbibnames = 9]{biblatex}
\bibliography{biblio}
\usepackage[title]{appendix}

\usepackage[most]{tcolorbox}

\def\E{\mathbb{E}}
\def\P{\mathbb{P}}

\newcommand{\DKL}{\text{D}_{\text{KL}}}
\renewcommand{\Re}{\text{Re}\,}
\renewcommand{\hat}{\widehat}

\DeclareMathOperator{\supp}{supp}
\newcommand{\h}{\text{h}}

\newcommand{\bbH}{\mathbb{H}}
\renewcommand{\subset}{\subseteq}
\renewcommand{\supset}{\supseteq}
\DeclareMathOperator{\Spec}{Spec}
\DeclareMathOperator{\LSpec}{LSpec}

\newtheorem{lemma}{Lemma}[section]
\newtheorem{theorem}[lemma]{Theorem}

\newtheorem{proposition}[lemma]{Proposition}
\newtheorem{Question}[lemma]{Question}
\newtheorem{conjecture}[lemma]{Conjecture}
\theoremstyle{definition}
\newtheorem{example}[lemma]{Example}
\newtheorem{definition}[lemma]{Definition}
\renewcommand{\d}{\text{d}}

\date{}

\title{Improved Bounds for the Freiman-Ruzsa Theorem}
\author{Rushil Raghavan}
\date{\small{Department of Mathematics, University of California, Los Angeles (UCLA) 
\\ 
e-mail: \href{mailto:rushil@math.ucla.edu}{rushil@ucla.edu}}}
\begin{document}
	\maketitle 
    
    \begin{abstract} Let $A$ be a finite subset of an abelian group $G$, and suppose that $|A+A|\leq K|A|$.
    We show that for any $\epsilon>0$, there exists a constant $C_\epsilon$ such that $A$ can be covered by at most $\exp(C_\epsilon \log(2K)^{1+\epsilon})$ translates of a convex coset progression with dimension at most $C_\epsilon \log(2K)^{1+\epsilon}$ and size at most $\exp(C_\epsilon \log(2K)^{1+\epsilon})|A|$. 
    This falls just short of the Polynomial Freiman-Ruzsa conjecture, which asserts that this statement is true for $\epsilon=0$, and improves on results of Sanders and Konyagin, who showed that this statement is true for all $\epsilon>2$. To prove this result, we use a mixture of entropy methods and Fourier analysis.
\end{abstract}
    
\section{Introduction}
This paper concerns the structure of approximate finite subgroups of an abelian group $G$. More specifically, given a finite subset $A\subset G$, if $A$ is ``approximately closed under addition" in the sense that 
\[|A+A| = |\{a+a':a,a'\in A\}| \leq K|A|,\]
then we seek to obtain a structural description of $A$ in terms of subgroups (or other algebraically structured subsets) of $G$. We call such sets $A$ \textit{sets with small doubling}, and say that $A$ has \textit{doubling at most} $K$. A simple example of a set with small doubling is a finite subgroup of $G$. If $G = \mathbb{Z}^d$, however, there are no nontrivial finite subgroups, but there are still sets with small doubling: if $Q$ is a convex body in $\mathbb{R}^d$, then $\mathbb{Z}^d\cap Q$ has doubling at most $\exp(O(d))$. 

It turns out that any set of small doubling can be ``built" out of these two examples. In the case $G = \mathbb{Z}^d$, this was first proved by Freiman (\cite{freiman1973foundations}, \cite{freiman1966nachala}), in the case of $G$ with bounded torsion by Ruzsa (\cite{RuzsaFiniteTorsion}), and for a general abelian group $G$, by Green and Ruzsa \cite{GRFreiman}:

\begin{theorem} Let $A\subset G$ have doubling at most $K$. Then $A$ can be covered by at most $O_K(1)$ translates of a convex coset progression\footnote{see Definition \ref{convprogression}} with dimension at most $O_K(1)$ and size at most $O_K(|A|)$. 
\end{theorem}

It is natural to try to quantify the dependence of the above statement on the parameter $K$, and thus determine how efficiently one can pass between the statistic $\frac{|A+A|}{|A|}$ and a strong algebraic description of the set $A$. 

\begin{Question}[Quantitative bounds for Green-Freiman-Ruzsa Theorems]\label{mainquestion} For which functions $f(K)$ is the following statement true?
Let $A$ be a finite subset of an abelian group $G$. Suppose that $|A+A|\leq K|A|$. Then $A$ can be covered by at most $\exp(f(K))$ translates of a convex coset progression with dimension at most $f(K)$ and size at most $\exp(f(K))|A|$?\end{Question}

The three appearances of $f(K)$ written above can be considered three different functions, and in some cases it is possible to decrease one of the bounds at the cost of increasing another. We will not need to do this, but a description of these maneuvers (and their tradeoffs) can be found in \cite[Section 12]{SRevisited}. 

The argument of Green and Ruzsa in \cite{GRFreiman} shows that Question \ref{mainquestion} is true for $f(K) = O(K^4\log^{O(1)}(2K))$. There have since been many improvements to these bounds. In \cite{sanders2012bogolyubov}, Sanders showed that one can take $f(K) = \log^{O(1)}(2K)$, and combining this with a bootstrapping argument by Konyagin, he established the best-known bound:
\begin{theorem}\label{Sandersbound}
    Question \ref{mainquestion} is true with $f(K) = O(\log^{3+o(1)}(2K))$. 
\end{theorem}

Our main result is an improvement upon this bound.

\begin{theorem}\label{improvedFreiman}Question \ref{mainquestion} is true for a function $f(K)$ with\footnote{The constant 3 is present to ensure that $f(K)$ is positive even when $K$ is very close to 1, and can otherwise be disregarded.} \[f(K) = O(\log(3K)\log\log(3K)).\] 
\end{theorem}

If $A$ is a union of $\exp(f(K))$ translates of a convex coset progression with dimension $f(K)$, then the doubling of $A$ is at most $\exp(O(f(K))$. It is thus natural to wonder if we can take $f(K) = O(\log(2K))$, so that we can obtain an equivalence between the data $\frac{|A+A|}{|A|}$ and the doubling of a set obeying the conclusions of Question \ref{mainquestion} up to factors polynomial in $K$. This is the content of the notorious \textit{Polynomial Freiman-Ruzsa conjecture}:

\begin{conjecture}[Polynomial Freiman-Ruzsa Conjecture]\label{PFR} Question \ref{mainquestion} is true for a function $f(K)$ with $f(K) = O(\log(2K))$. 
\end{conjecture}
Conjecture \ref{PFR} has a number of equivalent formulations; for example, it is equivalent to the corresponding statement where the set of all convex bodies is replaced with the set of ellipsoids \cite{Mannersformulations}. Another possible variation, where the set of all convex bodies is replaced with the set of all axis-aligned boxes, is false and was disproved by Lovett and Regev \cite{Lovett2017counterexample}. 

In the case where $G$ has bounded torsion, Conjecture \ref{PFR} was originally formulated by Marton, and proved by Gowers, Green, Manners, and Tao in \cite{GGMTBoundTorsion} and \cite{GGMTTwoTorsion}. The main result from \cite{GGMTBoundTorsion} is:
\begin{theorem} If $G$ has torsion at most $m$, then $A$ can be covered by at most $(2K)^{O(m^3)}$ cosets of a subgroup of $G$ with size at most $|A|$. 
\end{theorem}

Comparing this with Theorem \ref{improvedFreiman}, we see that our bound improves on theirs in the regime where $m^3$ is much larger than $\log\log(K)$.

One may also consider the variant of Question \ref{mainquestion} where the restriction on the size of the convex coset progression is lifted, so that one only cares about covering $A$ efficiently by translates of a low-dimensional set. Showing that $A$ can be covered by at most $\exp(O(\log(2K))$ translates of a convex coset progression with dimension at most $O(\log(2K))$ is sometimes called the \textit{weak Polynomial Freiman-Ruzsa conjecture}. In \cite{GMTRevisited}, it was shown that Marton's conjecture in $\mathbb{F}_2^n$ implies the weak Polynomial Freiman-Ruzsa conjecture, so this problem was solved in \cite{GGMTTwoTorsion}.

\subsection{Related Problems and Applications}
A very useful alternative structural result for sets of small doubling is one that provides a large structured set, such as a convex coset progression, contained in $A+A-A-A$. A result of this kind, for dense subsets of a finite group, was first proved by Bogolyubov \cite{bogoliouboff1939quelques} and was extended to arbitrary sets with small doubling by Ruzsa \cite{ruzsa1994generalized} en route to a proof of Freiman's theorem. The best known result in this direction is due to Sanders \cite{sanders2012bogolyubov}:
\begin{theorem}\label{Sandersbogolyubov} Let $A$ be a set with doubling at most $K$. Then $A+A-A-A$ contains a convex coset progression of dimension at most $O(\log^{6}(2K))$ and size at least $\exp(-O(\log^{6+o(1)}(2K)))|A|$. 
\end{theorem}

In most previous proofs of Freiman-Ruzsa type theorems, a Bogolyubov-Ruzsa type result is proved as an intermediate step. This is not the case for the arguments in \cite{GGMTBoundTorsion} and \cite{GGMTTwoTorsion}, whose arguments only find a structured set contained in a sum of $\log^{O_m(1)}(2K)$ copies of $A$. Our methods perform even more poorly at this task, and do not directly prove that there is a large convex coset progression inside $\underbrace{A+A+\dots+A}_{n\text{ times}}$ for any $n$. 

We now turn our attention towards applications of Theorem \ref{improvedFreiman}. One important application of structural results for sets with small doubling is an inverse theorem for the Gowers $U^3(G)$ norm, as proved by Green and Tao in \cite{GTinverse}. There, Green and Tao apply the Bogolyubov-Ruzsa lemma to prove the $U^3(G)$ inverse theorem. However, Samorodnitsky \cite{Samorodnitsky} showed that for $G=\mathbb{F}_2^n$, Marton's conjecture is sufficient to obtain polynomial bounds for the $U^3(G)$ inverse theorem. A similar implication when $G = \mathbb{F}_p^n$ for an odd prime $p$ was proved by Gowers, Green, Manners, and Tao in \cite[Appendix C]{GGMTBoundTorsion}. 

One may also modify the arguments in \cite{GTinverse} to obtain improved bounds on the dimension in the $U^3$ inverse theorem.

\begin{theorem}\label{U3Inverse} Let $G$ be a finite abelian group of odd order. Let $f:G\to \mathbb{C}$ be a one-bounded function and let $0<\eta\leq 1/4$. If $\|f\|_{U^3(G)}\geq \eta$, then there exists a regular Bohr set $B = B(S,\rho)$ with $|S| =  O(\log(\eta^{-1})^{1+o(1)})$ and $\rho  = \exp(-O(\log(\eta^{-1})^{1+o(1)})$, a translate $y\in G$, and a locally quadratic function $\phi:B+y\to\mathbb{R}/\mathbb{Z}$ such that 
\[\left|\frac{1}{|B|}\sum_{x\in B+y}f(x)e(-\phi(x))\right|\geq \exp(-O(\log(\eta^{-1})^{1+o(1)})).\]
\end{theorem}

We sketch a proof of Theorem \ref{U3Inverse} in Appendix \ref{U3appendix}. We note that the $\log(\eta^{-1})^{1+o(1)}$ terms can be replaced by $\log(\eta^{-1})\log\log^{O(1)}(\eta^{-1})$ if desired. However, we do not obtain a bound of exactly the same shape as that in Theorem \ref{improvedFreiman}, i.e., $\log(\eta^{-1})\log\log(\eta^{-1})$.

There are a number of other applications of Freiman-Ruzsa-type bounds, many mentioned in \cite{SRevisited}, such as $\Lambda(4) $ estimates for the squares and lower bounds for cardinalities $|A+\alpha A|$, where $A\subset \mathbb{R}$ is finite and $\alpha$ is transcendental. Our results provide no new information other than an improved input for these existing arguments, so we invite the reader to explore \cite[Section 13]{SRevisited} and the references therein for details.

Finally, as an application of Marton's conjecture, Prendiville \cite{Prendiville} proved an efficient inverse theorem for the Gowers $U^3(\mathbb{F}_p^n)$ norm relative to quadratic level sets, and then gave an application to estimating the Ramsey number of Brauer quadruples. The author would be interested to see if a similar result is possible for general finite abelian groups $G$, and whether bounds of the quality in Theorem \ref{improvedFreiman} can lead to good bounds for combinatorial configuration problems like those considered in \cite{Prendiville}.

\subsection*{Acknowledgements} 
An earlier version of this manuscript presented the bound \[f(K) = O(\log(K)\log\log(K)\log\log\log(K))\] in Theorem \ref{improvedFreiman}. Fredy Yip (private communication) observed that Proposition \ref{structureofminimizers} could be improved, with the eventual effect of removing the $\log\log\log(K)$ factor from this estimate. Special thanks to Fredy Yip for providing this optimization.

We would like to thank Terence Tao for a great deal of guidance and support. We would also like to thank James Leng for helpful conversations and encouragement, especially regarding the $U^3$ inverse theorem. Finally, we would like to thank Ben Green for helpful comments on a previous version of this manuscript. The author is partially supported by NSF Grant DMS-2347850.

\section{Notation and Strategy of Proof}
\subsection{Definitions and Notation}
To simplify notation in this paper, we introduce the following convention regarding iterated logarithms.
\begin{definition} For a real number $x\geq 3$, we define $\log_2(x) = \log\log(x)$. In particular, we never use $\log_2$ to refer to the base $2$-logarithm.
\end{definition}

\begin{definition}[Asymptotic notation] Given functions $F$, $G$, and $H$ from $\mathbb{N}$ to $[0,\infty)$, we say
\begin{itemize}
    \item $F(n) = O(G(n))$ if there is a constant $C>0$ such that $F(n)\leq CG(n)$ for all $n\in\mathbb{N}$,
    \item $F(n) = \Omega(G(n))$ if there is a constant $C>0$ such that $F(n)\geq CG(n)$ for all $n\in\mathbb{N}$,
    \item and $F(n) = \Theta(G(n))$ if $F(n) = O(G(n))$ and $F(n) = \Omega(G(n))$,
    \item $F(n) = o(G(n))$ if for all $c>0$, there is an $N_0\in\mathbb{N}$ such that for all $n\geq N_0$, $F(n)\leq cG(N)$,
    \item $F(n) = G(n)+O(H(n))$ if $|F(n) - G(n)| = O(H(n))$,
    \item $F(n) = G(n) + o(H(n))$ if $|F(n)-G(n)| = o(H(n))$. 
\end{itemize}
\end{definition}

Given a $G$-valued random variable $X$, we will let $p_X$ denote the probability density function of $X$, i.e., for any $x\in G$, $p_X(x) = \P(X=x)$. Given a finite subset $A\subset G$, we will let $U_A$ denote the random variable uniformly distributed on $A$.

Given a $G$-valued random variable $X$ and a positive integer $n$, we will let $nX$ denote the sum of $n$ iid copies of $X$. 

\begin{definition}[Shannon Entropy] Given a finitely supported $G$-valued random variable $X$, we define the shannon entropy $\bbH(X)$ by 
\[\bbH(X) = \sum_{x\in G}p_X(x)\log\left(\frac{1}{p_X(x)}\right),\]
with the convention that the above summand is zero if $p_X(x)=0$.
\end{definition}

\begin{definition}[Entropic Ruzsa Distance]
Let $X,Y$ be $G$-valued random variables. We define the entropic Ruzsa distance $d[X;Y]$ by 
\[d[X;Y] = \bbH(X'-Y')-\frac{1}{2}\bbH(X)-\frac{1}{2}\bbH(Y),\]
where $X'$ and $Y'$ are independent copies of $X$.
\end{definition}

We will need conditional variants of these quantities, which are defined as follows.
\begin{definition}[Conditional Entropy, Conditional Ruzsa distance] Given a $G$-valued random variable $X$ and another random variable $T$, we define 
\[\bbH(X|T) = \E_t \bbH(X|T=t)\]
and given another $G$-valued random variable $Y$ and another random variable $S$, we define 
\[\d[X|T;Y|S] = \E_{s,t}\d[X|T=t;Y|S=s].\]
\end{definition}

Entropy analogues of standard inequalities in additive combinatorics were developed by Ruzsa \cite{Ruzsaentropy} and Tao \cite{TSumsetEntropy} in 2009. Several authors have contributed to this program since then, developing an ``entropic Ruzsa calculus" which allows us to perform many of the standard operations from additive combinatorics in the entropic setting. We record the estimates from this theory which we will use in Appendix \ref{entropyappendix}.

\begin{definition}[Convex Progression, Convex Coset Progression]\label{convprogression} A set $P\subset G$ is called a $d$-dimensional convex progression if there is a centrally symmetric convex body $Q\subset\mathbb{R}^d$ and a homomorphism $\phi:\mathbb{Z}^d\to G$ such that $\phi(\mathbb{Z}^d\cap Q) = P$. 

A $d$-dimensional convex coset progression is the set of a form $H+P$, where $P$ is a $d$-dimensional convex progression and $H$ is a subgroup of $G$.
\end{definition}

Although the doubling constant and Ruzsa distance provide important information about the growth and dimension of a set $A\subset G$, they do not paint a complete picture of the approximate structure of $A$. We consider the following pair of standard examples, also discussed in \cite{SRevisited} and \cite{GGMTTwoTorsion}.

\begin{example} Let $M$ be a convex coset progression with dimension $O(\log(K))$ in $G$. We consider two rather different examples of sets of small doubling related to $M$. 

\begin{enumerate}
    \item Let $A_1$ be a subset of $M$ with density $K^{-1}$. Then $\d[U_{A_1};U_{A_1}]$ is $O(\log(K))$. 
    \item Let $A_2 = M\times \Gamma$, where $\Gamma$ is a linearly independent set of $K$ vectors in $\mathbb{Z}^K$. Then $\d[U_{A_2};U_{A_2}] = O(\log(K))$. 
\end{enumerate} 

Despite the fact that these two sets have approximately the same doubling constant, they satisfy Theorem \ref{improvedFreiman} in substantially different ways. $A_1$ is covered by one convex coset progression somewhat larger than $A_1$, while $A_2$ is covered by many translates of a convex coset progression somewhat smaller (in both size and dimension) than $A_2$. Although the doubling constant does not distinguish between $A_1$ and $A_2$, the entropies of iterated sums reveal a difference. As $n$ increases, $\bbH(nA_1)-\bbH(A_1) \approx \log(n)\log(K)$, but $\bbH(nA_2)-\bbH(A_2)\approx n\log(K)$.  
\end{example}

To distinguish between these instances of ``covering" and ``containment", we define the following, which is analogous to a combinatorial polynomial growth condition studied by Sanders in \cite{sanders2012bogolyubov}, \cite{SRevisited}.

\begin{definition}[Polynomial Growth]\label{entropypolygrowth} Let $X$ be a $G$-valued random variable. We say that $X$ has polynomial growth of order $d$ at scale $N$ if for all positive integers $n\leq N$, we have the inequality 
\[\bbH(nX)\leq \bbH(X)+d\log(n).\]
\end{definition}

\subsection{Strategy of Proof}

The proof of Theorem \ref{improvedFreiman} is split into two parts: the first uses information theory, and the second uses Fourier analysis. 

\textbf{Part 1: The Entropy Side.} To motivate the first half of our argument, we will give a brief overview of the proof of the Polynomial Freiman-Ruzsa conjecture over $\mathbb{F}_2$ given by Gowers, Green, Manners, and Tao \cite{GGMTTwoTorsion}. As in \cite{GGMTTwoTorsion}, we will obtain a structural characterization of the minimizers of a functional involving some entropic Ruzsa distances. More precisely, they prove the following.
\begin{theorem}[\cite{GGMTTwoTorsion}, Propositions 2.1 and 2.2]\label{GGMTTwoTorsionMainProp}
Let $X_1^0$ and $X_2^0$ be $\mathbb{F}_2^n$-valued random variables and let $X_1$, $X_2$ be $\mathbb{F}_2^n$-valued random variables which minimize the functional
\[d[X_1;X_2]+\frac{1}{9}d[X_1^0;X_1]+\frac{1}{9}d[X_2^0;X_2].\]
Then $d[X_1;X_2]=0$. In particular, there is a subgroup $H\subset \mathbb{F}_2^n$ such that $X_1$ and $X_2$ are each translates of $U_H$.
\end{theorem}

To prove Theorem \ref{GGMTTwoTorsionMainProp}, Gowers, Green, Manners, and Tao use a `Ruzsa distance decrement' argument. They show that if $\d[X_1;X_2]>0$, there are choices $X_1',X_2'$ for which $\d[X_1';X_2'] $ is smaller than $\d[X_1;X_2]$, but $\d[X_i^0;X_i']$ is not much larger than $\d[X_i^0;X_i]$ for $i=1,2$. Let us ignore the quantities $\d[X_i^0;X_i]$ for a moment and consider only the Ruzsa distance $\d[X_1;X_2]$. Using the fibring lemma (Lemma \ref{fibring}), they show the inequality
\[\d[X_1+Y_1;X_2+Y_2]+\d[X_1|X_1+Y_1;X_2|X_2+Y_2] \leq 2\d[X_1;X_2],\]
where $Y_i$ is an iid copy of $X_i$ for $i=1,2$. Thus, the sums $X_i+Y_i$ and the fibres $X_i|X_i+Y_i$ are candidates for a Ruzsa distance decrement. In the case where the above inequality\footnote{They also consider some additional similar inequalities related to ``sums" and ``fibers" of $X_1$ and $X_2$, but we will ignore this here.} is close to sharp, and no decrement can be obtained, they proceed to a different candidate to obtain a decrement, but the bounded torsion of $\mathbb{F}_2^n$ is required to make this work. However, in the integer setting, it remains true that if no choice $(X_1|X_1+Y_1=t_1,X_2|X_2+Y_2=t_2)$ provides a decrement, then $\d[2X_1,2X_2] \leq d[X_1,X_2]$. Elaborating on this, we may hope that in the absence of a decrement, the estimates 
\begin{equation}\label{fakepolygrowth}
    \d[2^nX_1;2^nX_2]\leq \d[X_1;X_2]
\end{equation}
hold for all $n\geq 1$. From this, it would follow that 
\[\bbH(2^nX_1 - 2^nX_2) \leq \bbH(X_1-X_2) + O(n\d[X_1;X_2])\]
for all $n$, which implies that $X_1-X_2$ has polynomial growth of order $O(\d[X_1;X_2])$. 

Of course, we cannot actually ignore the distances $\d[X_i^0;X_i]$, since if our goal is merely to minimize $\d[X_1;X_2]$, we could make this quantity zero by making each $X_i$ deterministic. Accounting for these terms worsens the inequality (\ref{fakepolygrowth}), and we thus obtain a slightly higher order of polynomial growth for the minimizers of our functional. In the following, let $A\subset G$ be a finite set with $|A+A|\leq K|A|$.

\begin{definition}[$\tau$ functional]\label{taufunctional} Let $\eta = 1/438$. For random variables $X,Y$ supported on $A$, we define 
\[\tau[X;Y] = \d[X;Y] + \frac{\eta}{\log_2(3K)}(\d[X;U_A]+\d[Y;U_A]).\]
\end{definition}

\begin{proposition}\label{structureofminimizers} Let $C$ be a constant. Suppose $(X,Y)$ are minimizers of $\tau$. Then either $K= O_C(1)$, or $X$ has polynomial growth of order $d=O(\log(3K))$ at scale $Cd^5$. Moreover, $\d[X;U_A]\leq 30d\log(d)$. 
\end{proposition}

The constants 1/438 and 30 above can certainly be improved, but we make no effort to do so.

\textbf{Part 2: The Fourier Side.} The arguments in this part of the proof largely mirror those by Sanders in \cite[Sections 9-11]{SRevisited}. Sanders considered a combinatorial polynomial growth condition instead of an entropic one, and showed that a set with polynomial growth of a given order could be related to a convex coset progression with controlled size and dimension. More specifically, he showed the following.

\begin{theorem}[\cite{SRevisited}, Theorem 2.7]\label{Sandersstructurefrompolygrowth} Let $X\subset G$ be a finite set such that 
\begin{equation}\label{combinatorialpolygrowth}   
|\underbrace{X+X+\dots+X}_{n\text{ summands}}|\leq n^d|X| \text{ for all }n\geq 1.\end{equation}
 Then there is a convex coset progression $M$ with dimension $O(d\log^2 d)$ and size at most $\exp(O(d\log^2d))|X|$ such that $X-X\subset M$. 
\end{theorem}

The condition (\ref{combinatorialpolygrowth}) is stronger than Definition \ref{entropypolygrowth} in the sense that if a set $X$ satisfies condition (\ref{combinatorialpolygrowth}) with an exponent $d$, then $U_X$ has polynomial growth of order $d$ as well. The converse implication is not true; see \cite[Appendix B]{GMTRevisited} for an example of large gaps between combinatorial and entropic notions of doubling. However, we are able to show using modifications of Sanders' arguments that under our weaker notion of polynomial growth, we may obtain a comparable amount of structure. 

\begin{proposition}\label{structurefrompolygrowth} There exists an absolute constant $C$ such that the following is true. Let $A$ be a finite set with $|A+A| = K|A|$. Suppose $X$ is a random variable supported on $A$ such that $X$ has polynomial growth of order $d$ at scale $Cd^5$ and\footnote{Here, the constant 30 may also be changed, only at the expense of the other implicit constants in the theorem. We make it explicit for convenience in section 4.} $\d[X;U_A] \leq 30d\log(d)$. Then there is a convex coset progression $M$ with dimension $O(d\log d)$ and size at most $\exp(O(d\log d))|A|$ such that $A$ can be covered by at most $K\exp(O(d\log(d)))$ translates of $M$. 
\end{proposition}

In view of the quantitative similarities between Theorem \ref{Sandersstructurefrompolygrowth} and Proposition \ref{structurefrompolygrowth}, it is natural to wonder how our argument provides any quantitative improvement over the results in \cite{SRevisited}. There, Sanders uses the Croot-Sisask lemma for almost-periods of convolutions \cite{CrootSisaskAP} to produce a somewhat sparse set $T$ ($\d[U_T;U_A]\approx \log(K)^3$) with a relatively high order of polynomial growth ($d\approx \log(K)^3$). The information-theoretic part of our argument produces a much ``denser" random variable with a much lower order of polynomial growth, which eventually leads to savings in the final bounds. The author would be very interested to see if the methods in this paper can be used to provide quantitative improvements to any of the many other applications of the Croot-Sisask lemma.  

Theorem \ref{improvedFreiman} follows straightforwardly from Propositions \ref{structurefrompolygrowth} and \ref{structureofminimizers}. We present the argument here:

\begin{proof}[Proof of Theorem \ref{improvedFreiman}, assuming Propositions \ref{structureofminimizers} and \ref{structurefrompolygrowth}] 
Let $A$ be a finite subset of an abelian group $G$ with $|A+A|\leq K|A|$. Let $\tau$ be the functional from Definition \ref{taufunctional}. We first observe that there is a minimizer, i.e., random variables $X,Y$ supported on $A$ such that $\tau[X;Y]\leq \tau[X';Y']$ for all random variables $X',Y'$ supported on $A$. Indeed, the space $P$ of probability distributions on $A$ is a compact subset of $\mathbb{R}^{A}$, and $\tau$ is a continuous function from $P\times P$ to $\mathbb{R},$ so there is a minimizer.

Let $C$ be the constant from Proposition \ref{structurefrompolygrowth}. We now select a minimizer $(X,Y)$ of $\tau$ and apply Proposition \ref{structureofminimizers}. If $K = O(1)$, we are done, since we can apply Theorem \ref{Sandersbound} to cover $A$ by $O(1)$ translates of a convex coset progression of size at most $O(|A|)$ and dimension at most $O(1)$. Otherwise, $X$ and $A$ satisfy the hypotheses of Proposition \ref{structurefrompolygrowth} with $d = O(\log(3K))$. There is thus a convex coset progression $M$ with dimension $O(d\log(d))$ and size at most $\exp(O(d\log(d))|A|$ such that $A$ can be covered by at most $K\exp(O(d\log(d)))$ translates of $M$. Since $d\log(d) = O(\log(3K)\log_2(3K))$, the claim follows.
\end{proof}

\section{Obtaining a Random Variable with Polynomial Growth}
In this section, we will prove Proposition \ref{structureofminimizers}. To clean up the presentation, we have chosen to make most constants in this section explicit. However, the precise values of these constants are not particularly important and are not optimized. 

\textbf{Remark.} In a previous version of this manuscript, Proposition \ref{growthofminimizers} was stated using a slightly different penalty term in (\ref{distanceofsums}). Fredy Yip (private communication) observed that the penalty term could be optimized, leading ultimately to a slight improvement to Theorem \ref{improvedFreiman}.

\begin{proposition}\label{growthofminimizers}  Let $(X,Y)$ be a minimizer of $\tau$. Let $\log(k) = \d[X,Y]$. Then either $K=O(1)$, or
\begin{equation}\label{distanceofsums}
    \d[nX;nY]\leq \log(k)+\frac{73\eta}{\log_2(3K)}\log(k)\log(n)
\end{equation}
for all $2\leq n\leq \log(K)^6$.
\end{proposition}
\begin{proof} We first consider the case $\log(k) = 0$. In this case, by Proposition \ref{subgroupdoubling}, $X$ and $Y$ are both translates of the uniform distribution on some subgroup $H\subset G$. Then $\d[nX;nY] = 0$ for all $n$, so (\ref{distanceofsums}) holds for all $2\leq n\leq \log(K)^6$.

From here on, we will assume $\log(k)>0$. Suppose $(X,Y)$ is a minimizer of $\tau$ but does not satisfy (\ref{distanceofsums}) for all $n\leq \log(K)^6$. Let $n$ be the smallest positive integer such that (\ref{distanceofsums}) fails. We then compute
\[\d[X;Y]+\d[(n-1)X;(n-1)Y]\leq 2\log(k)+\frac{73\eta}{\log_2(3K)}\log(k)\log(n-1).\]
By Proposition \ref{fibringapplication}, $\d[X;Y]+\d[(n-1)X;(n-1)Y]$ is bounded below by 
\[\d[nX;nY]+d[X|nX;Y|nY]> \log(k)+\frac{73\eta}{\log_2(3K)}\log(k)\log(n) + d[X|nX,Y|nY],\]
where by an abuse of notation, we write $X|nX=t$ to mean a random variable with distribution 
\[X_1|X_1+X_2+\dots+X_n = t\]
where the $X_i$ are independent and have the same distribution as $X$.
Thus 
\[\d[X|nX;Y|nY]\leq \log(k) - \frac{73\eta\log(k)}{\log_2(3K)}(\log(n)-\log(n-1))\leq \log(k)-\frac{73\eta\log(k)}{n\log_2(3K)}.\]

Let $m = n$ if $n$ is even and $m=n-1$ if $n$ is odd. We apply Proposition \ref{ERC}.(iii), then Proposition \ref{ERC}.(iv) twice, then Proposition \ref{ERC}.(ii) and then Proposition \ref{ERC}.(i) to obtain
\[\d[X|nX;U_A] -\d[X;U_A]\leq \frac{1}{2}(\bbH(nX)-\bbH((n-1)X)\leq \frac{1}{2}(\bbH(mX)-\bbH((m-1)X)\leq \]
\[\frac{1}{m}\sum_{j=m/2+1}^m\bbH(jX)-\bbH((j-1)X) = \frac{1}{m}\d[(m/2)X;-(m/2)X]\leq \]\[\frac{3}{m}\d[(m/2)X;(m/2)X]\leq \frac{6}{m}\d[(m/2)X;(m/2)Y].\]
Since (\ref{distanceofsums}) is true for $m/2<n$, this quantity is at most 
\[\frac{6}{m}\left(\log(k)+\log(k)\frac{73\eta\log(m/2)}{\log_2(3K)}\right),\] 
which is at most
\[\frac{36}{n}\log(k)\]
since $m\geq n/3$, $m\leq \log(K)^6$, and $6\cdot 73\eta \leq 1$.

By a similar argument, $\d[Y|nY;U_A] - \d[Y;U_A]\leq \frac{36}{n}\log(k)$. 

We thus have the inequalities 
\[\d[X|nX;Y|nY]  - \d[X;Y] \leq -\frac{73\eta\log(k)}{n\log_2(3K)},\]
\[\frac{\eta}{\log_2(3K)}(\d[X|nX;U_A] - \d[X;U_A]) \leq \frac{36\eta\log(k)}{n\log_2(3K)},\text{ and}\]
\[\frac{\eta}{\log_2(3K)}(\d[Y|nY;U_A] - \d[Y;U_A]) \leq \frac{36\eta\log(k)}{n\log_2(3K)}.\]
Adding these inequalities together, we find 
\[\sum_{s,t\in G}\mathbb{P}(nX=t,nY=s) \tau[X|nX=t;Y|nY=s] - \tau[X;Y] \leq -\frac{73\eta\log(k)}{n\log_2(3K)} + \frac{72\eta\log(k)}{n\log_2(3K)} < 0.\]
There thus exist $s,t\in G$ such that $\tau[X|nX=t;Y|nY=s]<\tau[X;Y]$, so $(X,Y)$ could not have been a minimizer of $\tau$.
\end{proof}

With this estimate complete, we are now ready to prove Proposition \ref{structureofminimizers}.

\begin{proof}[Proof of Proposition \ref{structureofminimizers}] Suppose $(X,Y)$ are minimizers of $\tau$. Let $\ell$ be any positive integer with $2^{\ell} \leq \log(K)^6.$ Then 
\[\bbH(2^\ell X) - \bbH(X) = \sum_{j=1}^\ell \bbH(2^jX)-\bbH(2^{j-1}X) = \sum_{j=1}^\ell \d[2^{j-1}X;-2^{j-1}X] \leq 6\sum_{j=1}^\ell \d[2^{j-1}X;2^{j-1}Y].\]
By Proposition \ref{growthofminimizers}, 
either $K=O(1)$, or each term is bounded above by 
\[\log(k)+\frac{73\eta}{\log_2(3K)}\log(2^j)\log(k) \leq \log(k)+\frac{6\cdot 73\eta\log_2(K)}{\log_2(3K)} \leq 2\log(k),\]
since $2^j\leq \log(K)^6$ and $6\cdot 73\cdot \eta \leq 1$.

We therefore have 
\[\bbH(2^\ell X)-\bbH(X)\leq 12\ell \log(k).\]

Thus, for any $n \leq \frac{1}{2}\log(K)^6$, if $\ell$ is the smallest positive integer with $2^\ell \geq n$, we have
\[\bbH(nX)-\bbH(X) \leq \bbH(2^\ell X) - \bbH(X) \leq 12\ell\log(k)\leq 24\log(n)\log(k).\]
Since $(X,Y)$ is a minimizer of $\tau$,
\[\log(k) = \d[X;Y]\leq \tau[X;Y]\leq \tau[U_A;U_A]\leq 3\log(K)\leq 3\log(3K).\]
Thus for any $n\leq \frac{1}{2}\log(K)^6$,
\[\bbH(nX)-\bbH(X) \leq 72\log(n)\log(3K).\]
Set $d = 72\log(3K)$. Since $d^5 = o(\log(K)^6)$, either $K=O_C(1)$, or 
\[\bbH(nX)-\bbH(X) \leq d\log(n)\]
for all $n\leq Cd^5$. 

It remains to show that $\d[X;U_A] \leq 30d\log(d)$. Since $(X,Y)$ minimizes $\tau$, we have 
\[\frac{\eta}{\log_2(3K)}\d[X;U_A]\leq \tau[X;Y]\leq \tau[U_A;U_A]\leq 3\d[U_A;U_A]\leq 3\log(3K),\]
so 
\[\d[X;U_A]\leq 3\eta^{-1}\log(3K)\log_2(3K)\leq 1500\log(3K)\log_2(3K).\]
This is at most $30d\log(d)$. 
\end{proof}

\section{Structure from Polynomial Growth}
In this section, we will use Fourier analysis to prove Proposition \ref{structurefrompolygrowth}. We will use several facts about the Fourier transform and Bohr sets, which are collected in Appendix \ref{Bohrappendix}. This section is very similar to \cite[Sections 9-10]{SRevisited}.

\begin{proposition}\label{smallgrowth} Let $X$ be a random variable with polynomial growth of order $d$ at scale greater than $32d\log(d)$. Let $n$ be the smallest positive integer greater than $32d\log(d)$. Then 
\[\bbH((n+1)X) - \bbH(nX)\leq \frac{1}{16}.\]
\end{proposition}
\begin{proof} By Proposition \ref{ERC}.(iv), 
\[\bbH((n+1)X) - \bbH(nX) \leq \frac{1}{n}\sum_{j=1}^n\bbH((j+1)X)-\bbH(jX) = \frac{1}{n}(\bbH((n+1)X) - \bbH(X)) \leq \frac{d\log(n+1)}{n},\]
which is at most $1/16$.
\end{proof}

\begin{proposition}\label{Bohrlowerbound} Let $X,Y$ be independent $G$-valued random variables with $\bbH(X+Y)-\bbH(Y) \leq \frac{1}{16}$ and suppose that $\epsilon\in (0,1)$ is a parameter. Let $B = B(\LSpec(Y,\epsilon),4\epsilon)$. Let $Z$ be another $G$-valued random variable. Then there is a set $S\subset \supp(X)$ such that $\log|S|\geq \bbH(Z) - 4\d[X;Z] - 2\log(2)$ and $S-S\subset B$.
\end{proposition}
\begin{proof} We will find a set $S\subset \supp(X)$ such that $\P(X\in S)\geq 1/2$ and $S-S\subset B$. The lower bound on $|S|$ will then follow from Proposition \ref{smalldoublingsetconcentration}. To begin, we apply Proposition \ref{SumEntropyIncreaseFormula} to find 
\[\frac{1}{16}\geq \sum_{x\in G}p_X(X)\DKL(x+Y||X+Y).\]
We then define $S = \{x\in \supp(X): \DKL(x+Y||X+Y)\leq 1/8\}$. Since the Kullback-Liebler divergence is always nonnegative, we may apply Markov's inequality and obtain $\P(X\in S)\geq 1/2$. By Pinsker's inequality, for any $x\in S$, $\|p_{x+Y}-p_{X+Y}\|_{\ell^1}\leq 1/2$. Thus, for any $x_1,x_2\in S$, $\|p_{x_1+Y}-p_{x_2+Y}\|_{\ell^1} \leq 1$. We then compute 
\[2 = \sum_{x\in G} \max(p_{x_1+Y}(x),p_{x_2+Y}(x))+\sum_{x\in G} \min(p_{x_1+Y}(x),p_{x_2+Y}(x)) \]
and 
\[1\geq \|p_{x_1+Y}-p_{x_2+Y}\|_{\ell^1} = \sum_{x\in G} \max(p_{x_1+Y}(x),p_{x_2+Y}(x))-\sum_{x\in G} \min(p_{x_1+Y}(x),p_{x_2+Y}(x)).\]
It follows that 
\[\sum_{x\in G} \min(p_{x_1+Y}(x),p_{x_2+Y}(x)) \geq 1/2.\]
We will now show that $S-S\subset B$. Let $x_1,x_2\in S$ and $\gamma\in \LSpec(Y,\epsilon)$. We wish to show that $|1-\gamma(x_1-x_2)|\leq 4\epsilon$. Write $\delta = 1-\sqrt{1-\epsilon^2/2}$. There is a phase $\omega\in S^1$ such that 
\[\sum_{x\in G}p_{Y}(x)\omega\overline{\gamma(x)} = |\hat{p_{Y}}(\gamma)| \geq 1-\delta.\]
Since the right hand side is real, we conclude that 
\[\sum_{x\in G}p_{Y}(x)\Re \omega\gamma(x) \geq 1-\delta.\]
In view of the identity $|1-z|^2 = 2-2\Re(z)$ for all $z\in S^1$, we have 
\[\sum_{x\in G}p_{Y}(x)|1-\omega\gamma(x)|^2 \leq 2\delta.\]
For $i\in\{1,2\}$, we have 
\[\sum_{x\in G}p_{Y+x_i}(x)|1-\omega\gamma(x_i)\gamma(x)|^2 = \sum_{x\in G}p_Y(x+x_i)|1-\omega\gamma(x_i+x)|^2 \leq 2\delta.\]
Observe also that for any $x\in G$, 
\[|1-\gamma(x_1-x_2)|^2 = |1-\omega\gamma(x_1)\gamma(x) + 1 - \omega\gamma(x_2)\gamma(x)|^2 \leq 2|1-\omega\gamma(x_1)\gamma(x)|^2+2|1-\omega\gamma(x_2)\gamma(x)|^2.\]
We thus have 
\[4\delta \geq \sum_{x\in G}p_{Y+x_1}(x)|1-\omega\gamma(x_1)\gamma(x)|^2+p_{Y+x_2}(x)|1-\omega\gamma(x_2)\gamma(x)|^2 \geq \]
\[\sum_{x\in G}\min(p_{Y+x_1}(x),p_{Y+x_2}(x)) (|1-\omega\gamma(x_1)\gamma(x)|^2+|1-\omega\gamma(x_2)\gamma(x)|^2) \geq  \]
\[\frac{1}{2}|1-\gamma(x_1-x_2)|^2\sum_{x\in G}\min(p_{x_1+Y}(x),p_{x_2+Y}(x)) \geq \frac{1}{4} |1-\gamma(x_1-x_2)|^2,\]
so $|1-\gamma(x_1-x_2)|^2 \leq 16\delta\leq 16\epsilon^2$.  
\end{proof}

\textbf{Remark.} The quantity $\sum_{x\in G}\min(p_{x_1+Y}(x),p_{x_2+Y}(x))$ is also used in a technique called Bernoulli part decomposition, which has been used to prove local central limit theorems. See \cite{gavalakis2024entropy}, \cite{mcdonald} for some examples.

\begin{proposition}\label{Bohrupperbound}
Let $U_A$ be the uniform distribution on a set $A$. Let $d$, $\epsilon\in(0,1)$, and $m,\ell\in\mathbb{N}$ be parameters such that $(1-\epsilon^2/2)^m \leq \frac{1}{2}\exp(-d\log(\ell m)-30d\log(d))$. Suppose also that we have  
\[\bbH(m\ell X+U_A) - \bbH(U_A) \leq d\log(m\ell)+30d\log(d).\]
Then\footnote{See Definition \ref{defLSpec} for the definition of $\LSpec$.}
\[|B(\LSpec(\ell X,\epsilon),1/2)|\leq 8|A|\exp(d\log(\ell m)+30d\log(d)).\]
\end{proposition}
\begin{proof} We begin by computing 
\[\int|\hat{p_X}(\gamma)|^{2\ell m}|\hat{p_{U_A}}(\gamma)|^2 = \|p_{\ell m X+U_A}\|_{\ell^2}^2 \geq \exp(-\bbH(\ell k X+U_A))\geq |A|^{-1}\exp(-d\log(m\ell)-30d\log(d)),\]
where in the first inequality we used Proposition \ref{Renyi2}. 
We also have 
\[\int_{\LSpec(\ell X,\epsilon)^c}|\hat{p_X}(\gamma)|^{2\ell m}|\hat{p_{U_A}}(\gamma)|^2\,d\gamma\leq (1-\epsilon^2/2)^{ m}\int|\hat{p_{U_A}}(\gamma)|^2\,d\gamma = \]
\[(1-\epsilon^2/2)^{m}|A|^{-1} \leq \frac{1}{2}|A|^{-1}\exp(-d\log(\ell m)-30d\log(d)).\]
We then have 
\[\int_{\LSpec(\ell X,\epsilon)}|\hat{p_X}(\gamma)|^{2\ell m}|\hat{p_{U_A}}(\gamma)|^2\,d\gamma\geq \frac{1}{2}|A|^{-1}\exp(-d\log(\ell m)-30d\log(d)). \]

Let $B$ be any finite subset of $B(\LSpec(X,\epsilon),1/2).$ Then 
\[\int|\hat{p_X}(\gamma)|^{2\ell m}|\hat{p_{U_A}}(\gamma)|^2|\hat{p_{U_B}}(\gamma)|^2\,d\gamma\geq \int_{\LSpec(\ell X,\epsilon)}|\hat{p_X}(\gamma)|^{2\ell m}|\hat{p_{U_A}}(\gamma)|^2|\hat{p_{U_B}}(\gamma)|^2\,d\gamma\geq\]\[2^{-2} \int_{\LSpec(\ell X,\epsilon)}|\hat{p_X}(\gamma)|^{2\ell m}|\hat{p_{U_A}}(\gamma)|^2\,d\gamma\geq  2^{-3}|A|^{-1}\exp(-d\log(\ell m)-30d\log(d)).\]
On the other hand, the first integral is $\|p_{\ell k X+U_A+U_B}\|_{\ell^2}^2\leq \|p_{U_B}\|_{\ell^2}^2 =  |B|^{-1}$. We thus have 
\[|B|\leq 8|A|\exp(d\log(\ell m)+30d\log(d)),\]
and since $B$ was an arbitrary finite subset of $B(\LSpec(X,\epsilon),1/2)$, the claim follows.
\end{proof}

With these estimates stored, we are now ready to prove Proposition \ref{structurefrompolygrowth}.
\begin{proof} Before providing the argument, we record the required parameters and estimates between them (and in particular, choose the constant $C$). Let $\ell$ be the least integer greater than $32d\log(d)$. We wish to choose $\epsilon\in (0,1)$ and positive integers $m,r$ so that 
\begin{enumerate}
    \item $(1-\epsilon^2/2)^m \leq \frac{1}{2}\exp(-d\log(\ell m)-30d\log(d))$,
    \item $\exp(d\log(\ell m)+150d\log(d)+\log(32))\leq 2^r$,
    \item $(3r+1)4\epsilon < 1/4$, 
    \item $\ell m \leq Cd^5$. 
\end{enumerate}

To solve (1), we can take $m = O(\epsilon^{-2}d\log(\epsilon^{-1} \ell d)) = O(\epsilon^{-2}d\log(\epsilon^{-1}d))$. To solve (3), we can take $r = \Theta(\epsilon^{-1})$. Plugging in these choices to (2), we wish to choose $\epsilon$ so that 
\[d\log(\ell \epsilon^{-2}d\log(\epsilon^{-1}d)) + 150d\log(d)+\log(32) \leq C_0\epsilon^{-1}\]
for some absolute constant $C_0$.
The left hand side is $O(d\log(d)+d\log(\epsilon^{-1}))$, so we may take $\epsilon^{-1} = O(d\log(d))$ to solve the inequality. 

With these choices of $\ell,m,r,\epsilon$, $\ell m = O(d^4\log^{O(1)}(d))$. Thus, provided $C$ is a large enough absolute constant, inequality (4) is satisfied as well.

We will now proceed to the main argument. Let $X$ be a random variable supported on a set $A$ such that $X$ has polynomial growth of order $d$ at scale $Cd^5$ and $\d[X;U_A]\leq 30d\log(d)$. By Proposition \ref{smallgrowth}, $\bbH(\ell X+X)-\bbH(\ell X) \leq \frac{1}{16}.$ Let $B_{\text{small}} = B(\LSpec(\ell X),\epsilon)$. Then, by Proposition \ref{Bohrlowerbound}, there is a set $S\subset A$ with \[\log|S|\geq \bbH(U_A)-4\d[X;U_A]-2\log(2) \geq \bbH(U_A) -120d\log(d)-2\log(2)\] and $S-S\subset B_{\text{small}}$. 
We will now verify the hypotheses required to apply Proposition \ref{Bohrupperbound}. Since $X$ has polynomial growth at scale $Cd^5\geq m\ell$, we have $\bbH(m\ell X)-\bbH(X)\leq d\log(\ell m).$ Thus, by Proposition \ref{ERC}.(iv), we have 
\[\bbH(m\ell X+U_A) - \bbH(X+U_A) \leq \bbH(m\ell X)-\bbH(X) \leq d\log(m\ell).\]

 Moreover, since $X$ is supported on $A$, we have $\bbH(X)\leq \bbH(U_A)$. We thus have  
\[\bbH(X+U_A) - \bbH(U_A) \leq \bbH(X+U_A) - \frac{1}{2}\bbH(U_A)-\frac{1}{2}\bbH(X) = \d[U_A;X] \leq 30d\log(d).\]
Thus $\bbH(m\ell X+U_A) - \bbH(U_A) \leq d\log(m\ell)+30d\log(d)$. Defining $B_{\text{large}} = B(\LSpec(\ell X,\epsilon),1/2)$, we then have 
\[|B_{\text{large}}|\leq |A|\exp(d\log(\ell m) + 30d\log(d) +\log(8)).\]
We then have \[\frac{|B(\LSpec(\ell X,\epsilon),(3r+1)4\epsilon)|}{|B_{\text{small}}|}\leq \frac{|B_{\text{large}}|}{|B_{\text{small}}|}\leq \exp(d\log(\ell m)+150d\log(d)+\log(32)) \leq 2^r\]
and $4\epsilon < \frac{1}{4(3r+1)}$, so by Theorem \ref{Bohrtoconvexcoset}, $B_{\text{small}}$ is an at most $r$-dimensional convex coset progression. 

We are now ready to record the desired conclusions and complete the proof. Set $M = B_{\text{small}}$. We have 
\[|M|\leq |B_{\text{large}}|\leq |A|\exp(O(d\log(d)),\]
so we have the desired size bound. We also have $r = O(d\log(d))$, so we have the desired dimension bound. Finally, we have 
\[|A+S|\leq |A+A|\leq K|A| \leq K\exp(O(d\log(d)))|S|.\]
By Ruzsa's covering lemma \cite[Lemma 2.14]{Tao_Vu_2006}, $A$ can be covered by at most $K\exp(O(d\log(d)))$ translates of $S-S$, and thus at most $K\exp(O(d\log(d)))$ translates of $M\supset S-S$. 
\end{proof}

\section{Barriers to Improved Bounds}
In this section, we record some observations about the quantitative limitations of our proof and speculate about what might be required to prove Conjecture \ref{PFR}. It is likely that all or almost all of the contents in this section are known to experts, but we hope that the reader benefits from seeing many of the relevant ideas and references collected in one place.

It has long been a strategy in additive combinatorics to prove something structural about a sumset $A+A$ (or a set $A$ with $|A+A|\leq K|A|$) by finding a set of translates $T\subset G$ such that for all $t\in T$, $A+A+t$ is similar $A+A$. We call such $t\in T$ ``almost-periods" of $A+A$. We may then learn something structural about $T$, and use that to deduce some information about the sumset $A+A$. Although the language is different, the proof of Theorem \ref{improvedFreiman} can be understood in a similar way. One summary of the proof is as follows.
\begin{enumerate}
    \item Use the fibring lemma to locate a ``dense" random variable $X$ supported on $A$ with a low order of polynomial growth.
    \item Observe that for a (somewhat large) value of $\ell$, $\bbH(X+\ell X) - \bbH(\ell X)$ is very small. Here, $X$ could be considered to be a random variable of ``almost-periods" for $\ell X$. 
    \item Find a set $S\subset \supp(X)$ such that for all $x_1,x_2\in S$, $\|p_{\ell X + x_1}-p_{\ell X+x_2}\|_{1} $ is very small. Here, $S$ could be considered to be a set of almost-periods for $\ell X$ in a way that more closely resembles traditional arguments in additive combinatorics, such as in \cite{Bourgain_sumsetAPs} or \cite{CrootSisaskAP}. 
    \item Use Fourier analysis to relate $S$ to a ``structured object", in this case a Bohr set and then a convex coset progression. 
\end{enumerate}

Throughout, we make essential use of both upper bounds on $\bbH(\ell X)$ and lower bounds on $\bbH(X)$ (the latter implicitly via upper bounds on $\d[X;U_A]$). However, the gap between $\bbH(\ell X)$ and $\bbH(X)$ is larger than $O(\log(K))$. In other words, when considering these kinds of almost-periods, we either need to interpret a large subset of $A$ as a set of almost-periods for a much larger random variable, or find a much sparser random variable as a set of almost-periods of a random variable close to $A$.

This phenomenon occurs even when we restrict our attention to idealized examples such as Gaussians and uniform distributions on convex progressions. We will consider the problem of finding $d$-dimensional discrete Gaussians $Z_1,Z_2$ with $\bbH(Z_1+Z_2) - \bbH(Z_1) = O(1)$. Let $Z_{\sigma^2}$ be a discretized\footnote{There are many slightly different ways to discretize a Gaussian, but this does not matter much in this discussion. For concreteness, in $\mathbb{Z}^d$, you may consider a tuple of $d$ independent centered Binomial$(n,1/2)$ distributions for a large $n$.} Gaussian distribution in $\mathbb{Z}^d$ with covariance matrix $\begin{pmatrix}\sigma^2 & 0 & \hdots & 0 \\ 0 & \sigma^2 & \hdots & 0 \\ \vdots & \vdots & \ddots & 0 \\ 0 & 0 & 0 & \sigma^2\end{pmatrix}$.

We then have $\d[Z_{\sigma_1^2};Z_{\sigma_1^2}] = O(d)$. Also,
\[Z_{\sigma_1^2+\sigma_2^2} \approx Z_{\sigma_1^2}+Z_{\sigma_2^2}\quad\text{ and }\bbH(Z_{\sigma_1^2+\sigma_2^2})-\bbH(Z_{\sigma_1^2}) \approx \frac{d}{2}\log(1+\sigma_2^2/\sigma_1^2).\]
Thus, in order to obtain $\bbH(Z_{\sigma_1^2}+Z_{\sigma_2^2}) - \bbH(Z_{\sigma_1^2}) = O(1)$, we would need to take $\sigma_2^2 = O(d^{-1}\sigma_1^2)$, but then $\bbH(Z_{\sigma_1^2}) - \bbH(Z_{\sigma_2^2}) = \Omega(d\log(d))$. A similar estimate holds when considering uniform distributions on $d$-dimensional convex progressions instead of Gaussians.

This observation does not, on its own, prohibit the possibility of using almost-periods as part of a proof of Conjecture \ref{PFR}. However, we believe that without any new ideas, the methods used in this paper are unlikely to yield a bound better than Theorem \ref{improvedFreiman}, i.e., $f(K) = O(\log(K)\log\log(K))$, for Question \ref{mainquestion}. 

We now turn our attention to a more speculative topic: what kinds of results would be required to prove Conjecture \ref{PFR}? Before discussing this, let us again discuss the Gowers-Green-Manners-Tao proof of the conjecture in bounded torsion \cite{GGMTTwoTorsion}, \cite{GGMTBoundTorsion}. Although the fibring lemma and its immediate consequences are still true for $G = \mathbb{Z}^d$, there are two things present in the bounded torsion case that are missing in the case of unbounded torsion. Both issues are barriers to running a ``Ruzsa distance decrement" style argument in general.

\textbf{Barrier 1: The ``endgame" step.} if none of your operations, such as passing to sums $X_1+X_2$ or fibers $X_1|X_1+X_2$ provide a Ruzsa distance decrement, what do you do? In the case of bounded torsion, Gowers-Green-Manners-Tao are able to find another operation that ultimately reduces the Ruzsa distance. This is not possible in $\mathbb{Z}^d$: a discrete Gaussian has doubling $\Theta(d)$, which may be very large, and the class of $d$-dimensional Gaussians is stable under both of these operations (as well as many other natural operations one may consider). Moreover, Gaussians in $\mathbb{R}^d$ minimize the differential entropic doubling constant, and it is likely (though not known, except for when $d=1$ \cite[Theorem 1.13]{TSumsetEntropy}) that they are approximate minimizers in $\mathbb{Z}^d$ as well. Thus, the ``base case" of a Ruzsa distance induction argument may not necessarily occur when $\d[X_1;X_2] = O(1)$, and whatever argument occurs in a base case must be efficient enough to obtain polynomial bounds for Question \ref{mainquestion}. 

\textbf{Barrier 2: A stable characterization of Ruzsa distance extremizers.}In \cite{GMTRevisited}, Green-Manners-Tao prove that if $\d[X;Y] \leq \epsilon$ and $\epsilon$ is sufficiently small, then there is a subgroup $H$ of $G$ such that $\d[X;U_H]$ and $\d[Y;U_H]$ are both at most $12\epsilon.$ This serves as the base case of the Ruzsa distance decrement argument proving Marton's conjecture in groups with bounded torsion. 

Tao proved \cite[Theorem 1.13]{TSumsetEntropy} that in $\mathbb{Z}$, for random variables with sufficiently high entropy, that the minimal possible entropic doubling constant is approximately $\frac{1}{2}\log(2)$, which matches the behavior of Gaussians. Let us call a $\mathbb{Z}^d$-valued random variable ``genuinely $d$-dimensional" if it is not concentrated on any hyperplane in $\mathbb{Z}^d$. A possible program for proving Conjecture \ref{PFR} would be to carry out the following two steps:
\begin{enumerate}
    \item Show that a genuinely $d$-dimensional random variable with sufficiently high entropy has doubling at least $\frac{d}{2}\log(2)-\epsilon$. 
    \item Show that a genuinely $d$-dimensional random variable with nearly minimal doubling is ``close" to a Gaussian (eventually in Ruzsa distance, but perhaps in some other distance to start).  
\end{enumerate}

Such a program could yield a proof of Conjecture \ref{PFR}, and part (1) of the conjecture may be within reach. But part (2) is more difficult, and meets some subtle issues in continuous information theory. Given an $\mathbb{R}^d$-valued random variable $X$, we define the differential entropy by 
\[\h(X) = -\int_{\mathbb{R}^d}p_X(x)\log(p_X(x))\,dx.\]
It is then true that $\h(X_1+X_2)\geq \frac{1}{2}\h(X_1)+\frac{1}{2}\h(X_2)+\frac{d}{2}\log(2)$, with equality if and only if $X_1,X_2$ are Gaussian \cite{STAM1959101}. This is known as the Shannon-Stam inequality. However, the question of stability of this inequality is subtle. Courtade, Fathi, and Pananjady \cite{CourtadeFathiPananjady} have shown that for all $\epsilon>0$, there exist $\mathbb{R}$-valued random variables $X_1,X_2$ such that $\h(X_1+X_2)\leq \frac{1}{2}\h(X_1)+\frac{1}{2}\h(X_2)+\frac{1}{2}\log(2)+\epsilon$ but that for any Gaussian random variable $Z$, $W_2^2(X_i,Z)\geq 1/3$ for $i\in\{1,2\}$, where $W_2$ is the Wasserstein 2-distance. The counterexamples are (heavy-tailed) Gaussian mixtures, and they are still close to Gaussians in Ruzsa distance, but this demonstrates that there are limitations to the kinds of methods that can be used to prove stability of the Shannon-Stam inequality in general. There are many results that prove stability under strong hypotheses on the random variables $X_1$, $X_2$, such as log-concavity (also proved in \cite{CourtadeFathiPananjady}), but it seems unlikely that these kinds of exact conditions can be met by the approximate extremizers obtained in a Ruzsa distance decrement argument. For a survey of some stability estimates for the Shannon-Stam inequality, see \cite[Section 2.2.2]{MadimanMelbournePeng}. It is worth mentioning that in the discrete setting $G=\mathbb{Z}$, Gavalakis and Kontoyiannis \cite{gavalakis2025conditionsequalitystabilityshannons} have obtained a stablility estimate for the discrete Shannon-Stam inequality, but only when the random variables involved are log-concave. 

Instead of the Shannon-Stam inequality, one may consider other characterizations of Gaussian distributions. One particularly attractive one is the Kac-Bernstein theorem: if $X$ and $Y$ are $\mathbb{R}$-valued random variables such that $X+Y$ and $X-Y$ are independent, then $X$ and $Y$ are Gaussian. A number of stability estimates \cite{lukacs1977stability}, \cite{bernsteinstability} for this inequality are known, quantifying weak dependence using various metrics, but we do not see at present how to discretize these results and relate them to Ruzsa distance.

\begin{appendices}
    \section{Entropy and Additive Combinatorics}\label{entropyappendix}
    In this appendix, we will record some standard inequalities regarding Shannon entropy, as well as some information-theoretic analogs of standard inequalities in additive combinatorics. 

We begin by writing two consequences of Jensen's inequality.
\begin{proposition}\label{finitesetsupport} Let $X$ be a random variable supported on a finite set $A$. Then $\bbH(X)\leq \log|A|$.
\end{proposition}

\begin{proposition}\label{Renyi2} Let $X$ be a $G$-valued random variable. Then $-\log\|p_X\|_{\ell^2}^2 \leq \bbH(X).$

\end{proposition}
\begin{proof} The function $t\mapsto -\log(t)$ is a convex function, so by Jensen's inequality,
\[-\log\left[\sum_{x\in G}p_X(x)^2\right]\leq \sum_{x\in G}-p_X(x)\log(p_X(x)). \qedhere\]
\end{proof}
\textbf{Remark.} The quanties $-\log(\|p_X\|_{\ell^2}^2)$ and $\log|\supp(X)|$ are known as the Renyi 2-entropy $H_2$ and Renyi 0-entropy $H_0$, respectively. In this framework, the Shannon entropy is the Renyi 1-entropy $H_1$. The previous two propositions can then be written as the inequality $H_2\leq H_1\leq H_0$.

\begin{definition}[Kullback-Leibler Divergence] For $G$-valued random variables $X,Y$, we define 
\[\DKL(X||Y) = \sum_{x\in G}p_X(x)\log\left(\frac{p_X(x)}{p_Y(x)}\right),\]
with the convention that the summand is zero if $p_X(x)=0$ and infinite if $p_X(x)\neq 0$ but $p_Y(x)=0$. 
\end{definition}

\begin{theorem}[Pinsker's Inequality]\label{Pinsker} For $G$-valued random variables $X,Y$, we have 
\[\|p_X-p_Y\|_{\ell^1}^2\leq 2\DKL(X||Y).\]
\end{theorem}

For a proof, see \cite[page 132]{pinsker}.

\begin{proposition}\label{SumEntropyIncreaseFormula}
If $Y,Z$ are $G$-valued random variables, we have 
\[\bbH(Z-Y)-\bbH(Y)= \sum_{z\in G}p_Z(z)\DKL(z-Y||Z-Y).\] 
\end{proposition}
This is the equality case of \cite[equation A.18]{GMTRevisited}. 

    \begin{proposition}\label{ERC}
    Let $G$ be an abelian group, and let $X,Y,Z$ be $G$-valued random variables. The following statements are all true.
    \begin{enumerate}
        \item[(i)] $\d[X;Y]=\d[Y;X]\geq 0$ and $\d[X;Z]\leq \d[X;Y]+\d[Y;Z]$. In particular, $\d[X;X]\leq 2\d[X;Y]$. 
        \item[(ii)] $\d[X;-Y]\leq 3\d[X;Y]$.
        \item[(iii)] If $Y,Z$ are independent, then $\d[X;Y|Y+Z]\leq \d[X;Y]+\frac{1}{2}(\bbH(Y+Z)-\bbH(Z))$.
        \item[(iv)] If $X,Y,Z$ are independent, then $\bbH(X+Y+Z)-\bbH(X+Y)\leq \bbH(Y+Z)-\bbH(Y)$.

        \item[(v)] $|\bbH(X)-\bbH(Y)|\leq 2\d[X;Y]$.
    \end{enumerate}
    \end{proposition}
    \begin{proof} Statements(i) and (ii) are from \cite[Theorem 1.10]{TSumsetEntropy}. Statement (iii) is \cite[Lemma 5.2]{GGMTTwoTorsion}. Statement (iv) is \cite[Theorem I]{madiman2008entropy}. Statement (v) follows from subtracting 
    \[\bbH(X-Y)-\frac{1}{2}\bbH(X)-\frac{1}{2}\bbH(Y) = \d[X;Y]\]
    from 
    \[\bbH(X-Y)-\bbH(X) \geq 0\]
    to prove $\bbH(X)-\bbH(Y) \leq 2\d[X;Y]$, and then using a symmetric argument to show $\bbH(Y)-\bbH(X)\leq 2\d[X;Y]$. 
    \end{proof}

    We have seen that $\d[X;Y]\geq 0$. Equality holds if and only if $X,Y$ are both translates of the uniform distribution on a subgroup $H$ of $G$, as proved in \cite[Theorem 1.11.(i)]{TSumsetEntropy}.
    \begin{proposition}\label{subgroupdoubling} $\d[X;Y] = 0$ if and only if there is a subgroup $H\subset G$ and translates, $s,t\in G$ such that $X$ has the same distribution as $s+U_H$ and $Y$ has the same distribution as $t+U_H$.
    \end{proposition}

    If a random variable $X$ has small doubling, then it cannot be concentrated on a very small set. The following estimate is not explicitly stated in the literature, but it is used as a step in the proof of \cite[Proposition 1.2]{GMTRevisited}. 

    \begin{proposition}\label{smalldoublingsetconcentration} Let $X$ be a $G$-valued random variable. Let $S\subset G$ be a set with $\P(X\in S) \geq \frac{1}{2}$. Then for any other $G$-valued random variable $Y$,
    \[\log|S| \geq \bbH(Y) - 4\d[X;Y]-2\log(2).\]
    \end{proposition}
    \begin{proof} Let $A$ be the random variable which takes the value $1$ if $X\in S$ and $0$ otherwise. Let $p = \P(X\in S) \geq 1/2$. Then 
    \[\bbH(X) = \bbH(X|A) + \bbH(A) = p\bbH(X|A=1)+(1-p)\bbH(X|A=0) + \bbH(A).\]
    Rearranging this equation, we find 
    \begin{equation}\label{outsideentropy} (1-p)\bbH(X|A=0) = \bbH(X) - p\bbH(X|A=1) - \bbH(A).
    \end{equation}
    For any $i\in\{0,1\}$, if $Y$ is a $G$-valued random variable independent of $X$, we have
    \[\bbH(X-Y|A=i)\geq \bbH(Y)\quad\text{and}\quad \bbH(X-Y|A=i) \geq \bbH(X|A=i).\]
    We thus have 
    \[\bbH(X-Y)\geq \bbH(X-Y|A) = p\bbH(X-Y|A=1) + (1-p)\bbH(X-Y|A=0) \geq\]
    \[p\bbH(Y)+\frac{1-p}{2}(\bbH(Y)+\bbH(X|A=0)).\]
    Substituting (\ref{outsideentropy}), we find 
    \[\bbH(X-Y)\geq \frac{1}{2}\bbH(Y) + \frac{p}{2}\bbH(Y) + \frac{1}{2}\bbH(X) - \frac{p}{2}\bbH(X|A=1) - \frac{1}{2}\bbH(A).\]
    Rearranging this inequality, we find 
    \[\frac{p}{2}\bbH(X|A=1) \geq -\d[X;Y] + \frac{p}{2}\bbH(Y) - \frac{1}{2}\bbH(A),\]
    so 
    \[\bbH(X|A=1)\geq -\frac{2}{p}\d[X;Y] + \bbH(Y) - \frac{1}{p}\bbH(A) \geq \bbH(Y) - 4\d[X;Y] - 2\log(2).\]
    Since $X|A=1$ is supported on $S$, we have $\bbH(X|A=1)\leq \log|S|$, which provides the desired estimate.
    \end{proof}

    Given the close relationship between the estimates in Proposition \ref{ERC} and corresponding combinatorial inequalities well-known in additive combinatorics, one may wonder what the advantage is of working in an entropic framework. One answer to this question is that the entropy notions are considerably better behaved under homomorphisms. This was originally proved in \cite{GMTRevisited}, and is the driving force behind the solution to Marton's conjecture in abelian groups with bounded torsion \cite{GGMTTwoTorsion}, \cite{GGMTBoundTorsion}. We record the following estimate, which is \cite[Proposition 1.4]{GMTRevisited}. A refinement of this result (explicitly calculating the difference between both sides of the inequality) is used in \cite{GGMTTwoTorsion} and \cite{GGMTBoundTorsion}, but we do not need this error term here.

    \begin{lemma}[Fibring Lemma]\label{fibring} Let $\pi:H\to H'$ be a homomorphism between abelian groups and let $Z_1,Z_2$ be $H$-valued random variables. Then we have 
    \[\d[Z_1;Z_2]\geq \d[\pi(Z_1);\pi(Z_2)] +\d[Z_1|\pi(Z_1);Z_2|\pi(Z_2)]. \]
    \end{lemma}

    We record one application of the fibring lemma for use in our argument.
    \begin{proposition}\label{fibringapplication} Let $X,Y$ be $G$-valued random variables. For a positive integer $n$, let $nX$ denote the sum of $n$ independent copies of $X$, and similarly for $Y$. Then for any integer $n\geq 2$, 
    \[\d[X;Y]+\d[(n-1)X;(n-1)Y] \geq \d[nX;nY] + \d[X|nX;Y|nY].\]
    \end{proposition}
    \begin{proof} Let $X'$ have the same distribution as $X$ and be independent from $(n-1)X$, and define $Y'$ similarly. The left-hand side of the desired inequality is then equal to 
    \[\d[(X',(n-1)X);(Y',(n-1)Y)].\]
    We apply Lemma \ref{fibring} with $H = G\times G$ and $\pi:G\times G\to G$ the addition homomorphism $(u,v)\mapsto u+v$. Then $\pi(X',(n-1)X)$ is distributed as $ nX$, and $\pi(Y',(n-1)Y)$ is distributed as $nY$. We then have 
    \[\d[X;Y]+\d[(n-1)X;(n-1)Y] \geq \d[nX;nY] + \d[(X',(n-1)X|X'+(n-1)X;(Y',(n-1)Y)|Y'+(n-1)Y].\]
    Since $X'+(n-1)X$ and $X'$ determine $(n-1)X$ and $Y'+(n-1)Y$ and $Y'$ determine $(n-1)Y$, we have 
    \[\d[(X',(n-1)X|X'+(n-1)X;(Y',(n-1)Y)|Y'+(n-1)Y] = \]\[ \d[X'|X'+(n-1)X;Y'|(n-1)Y] = \d[X|nX;Y|nY]. \qedhere\]
    \end{proof}
    
    \section{Fourier Analysis and Bohr Sets}\label{Bohrappendix}     
    
    Given a discrete group $G$ we write $\hat{G}$ for the set of homomorphisms $G\to S^1$ where $S^1=\{z\in\mathbb{C}:|z|=1\}$.        
    Since $G$ is discrete, $\hat{G}$ is a compact group with a translation-invariant probability measure. Given $f\in\ell^1(G)$ we define $\hat{f}:\hat{G}\to\mathbb{C}$ by 
    \[\hat{f}(\gamma)=\sum_{x\in G}f(x)\overline{\gamma(x)}\]
    for all $\gamma\in\hat{G}$. 

    \begin{proposition}[Convolution-to-product identity] For a $G$-valued random variable $X$, we have $\hat{p_{X-X}} = |\hat{p_X}|^2$. 
    \end{proposition}

    \begin{proposition}[Plancherel's formula] For functions $f,g:G\to\mathbb{C}$, we have 
    \[\sum_{x\in G}f(x)\overline{g(x)} = \int_{\hat{G}}\hat{f}(\gamma)\overline{\hat{g}(\gamma)}\,d\gamma.\]
    \end{proposition}

    \begin{proposition}[Fourier Inversion] For a function $f:G\to\mathbb{C}$, we have 
    \[f(x) = \int_{\hat{G}} \hat{f}(\gamma)\gamma(x)\,d\gamma.\]
    \end{proposition}

We will now introduce Bohr sets, which are a type of additively structured set ubiquitous in additive combinatorics. Much of the standard material about Bohr sets can be found in \cite[Chapter 4]{Tao_Vu_2006}.

    \begin{definition}[Bohr Set] Given a neighborhood of characters $\Gamma\subset \hat{G}$ and a parameter $\delta\in(0,2]$ we define the Bohr set $B(\Gamma,\delta)$ by 
    \[\{x\in G:|\gamma(x)-1|\leq \delta\text{ for all }\gamma\in\Gamma\}.\]
    
    \end{definition}

    We will define some important sets of characters, which we will use to generate the Bohr sets relevant to our argument. 
    \begin{definition}[99\% Large Spectrum]\label{defLSpec} Given a $G$-valued random variable $X$ and a parameter $\epsilon\in(0,1)$, we define \[\LSpec(X,\epsilon) = \{\gamma\in\hat{G}:|\hat{p_X}(\gamma)|^2\geq 1-\epsilon^2/2\}.\]
    \end{definition}

    In many parts of the literature, it is important to consider a somewhat different collection of characters. We record this alternate definition for use in some results required in Appendix \ref{U3Inverse}.

    \begin{definition}[1\% Large Spectrum] Given a $G$-valued random variable $X$ and a parameter $\epsilon\in(0,1)$, we define 
    \[\Spec(X,\epsilon) = \{\gamma\in\hat{G}:|\hat{p_X}(\gamma)|\geq \epsilon\}.\]
    \end{definition}

    To understand some of the numerology surrounding the sizes of Bohr sets, it is helpful to restrict our attention to the case where $G$ is a finite group. 

    \begin{proposition}\label{bohrsizebounds}[Size Bounds] If $G$ is a finite group, $S\subset \hat{G}$ a set of characters, and $\rho>0$, then
    \begin{itemize}
        \item $|B(S,\rho)|\geq \rho^{|S|}|G|$
        \item $|B(S,2\rho)|\leq 4^{|S|}|B(S,\rho)|$.
    \end{itemize}
    \end{proposition}
    We think of $|S|$ as corresponding to the dimension of the Bohr set. This is related to the dimension of a convex coset progression, as we can see in the following example:
    \begin{example} Let $G = (\mathbb{Z}/N\mathbb{Z})^d$. Let $e^1,\dots,e^d$ be the vectors in $G$ with $e^i_j = 1$ if $i=j$ and $0$ if $i\neq j$. 
    
    For any $j\in\{1,\dots,d\}$, let $\gamma_j:(\mathbb{Z}/N\mathbb{Z})^d\to S^1$ be the homomorphism with \[\gamma_j(e^i) = \begin{cases} e^{2\pi i /N} & i = j \\ 1 & i \neq j.
    
    \end{cases}\]

    Consider the Bohr set $B(\{\gamma_1,\dots,\gamma_d\}, 1/100)$. Then $(x_1,\dots,x_d)\in B$ if and only if $x_i \in \left[-N\frac{\arccos(1-1/20000)}{2\pi}, N\frac{\arccos(1-1/20000)}{2\pi}\right]$ for all $i$. 

    Thus a $d$-dimensional cube in $\mathbb{Z}^d$ can be thought of as a Bohr set with $|S|=d$. 
    \end{example}

    In view of this example, it would be useful if we could pass from a Bohr set with statistics like those in Proposition \ref{bohrsizebounds} to a convex coset progression with controlled dimension. 
    The following is \cite[Proposition 9.1]{SRevisited}, and it allows us to do so efficiently. 
    
    \begin{theorem}\label{Bohrtoconvexcoset} Suppose $B(\Gamma,\delta)$ is a finite Bohr set and $r\in\mathbb{N}$ is such that $|B(\Gamma,(3r+1)\delta)|\leq 2^r|B(\Gamma,\delta)|$ for some $\delta < 1/(4(3r+1)).$ Then $B(\Gamma,\delta)$ is an at most $r$-dimensional convex coset progression.
    \end{theorem}
    
    We will also record some facts about Bohr sets for use in Appendix \ref{U3appendix}.

    For the remainder of the appendix, $G$ will be a finite group. Given a Bohr set $B(S,\rho)$, we will refer to $|S|$ as the \textit{rank} of a Bohr set and $\rho$ as the \textit{radius} of a Bohr set.

    \begin{definition}[Regular Bohr set] A Bohr set $B(S,\rho)$ of rank $d$ is \textit{regular} if the estimate
    \[(1-100d|\kappa|)|B(S,\rho)|\leq |B(S,(1+\kappa)\rho)|\leq (1+100d|\kappa|)B(S,\rho)\]
    for all $\kappa$ with $|\kappa|\leq \frac{1}{100d}$.
    \end{definition}

    Not all Bohr sets are regular, but they can always be modified slightly to obtain a regular Bohr set. The following is \cite[Lemma 4.25]{Tao_Vu_2006}.
    \begin{lemma} Let $S\subset \hat{G}$ be nonempty and let $\epsilon\in(0,1)$. Then there is a $\rho\in[\epsilon,2\epsilon]$ such that $B(S,\rho)$ is regular.
    \end{lemma}

    The 1\% large spectrum of a subset of an abelian group has some additive structure. We record two variants of this fact: one over the entire group $G$, and one localized to a regular Bohr set. The following two are \cite[Lemma 4.36]{Tao_Vu_2006} and \cite[Lemma 5.3]{SchoenSisask}, respectively.

    \begin{lemma}[Global Chang's Lemma]\label{globalchang} Let $A\subset G$ and let $\alpha = \frac{|A|}{|G|}$. Let $\epsilon \in(0,1]$ and $\nu\in(0,1]$ be parameters. There is a $d = O(\epsilon^{-2}(\log(2/\alpha))$, a $\rho  = \Omega(\nu/d)$, and a Bohr set $B$ with rank $d$ and radius $\rho$ such that 
    \[|1-\gamma(t)|\leq \nu\quad\text{for all}\quad \gamma\in\Spec_\epsilon(p_{U_A}) \text{ and }t\in B.\]
    \end{lemma}

    \begin{lemma}[Local Chang's Lemma]\label{localchang} Let $\epsilon,\nu\in(0,1]$. Let $B = B(S,\rho)$ be a regular Bohr set of rank $d$ and let $A\subset B$. Let $\alpha = \frac{|A|}{|B|}$. Then there is a set of characters $T\subset\hat{G}$ and a radius $\rho'$ with 
    \[|T| = O(\epsilon^{-2}\log(2/\alpha))\quad\text{and}\quad \rho' = \Omega(\rho\nu\epsilon^2d^{-2}/\log(2/\alpha))\]
    such that 
    \[|1-\gamma(t)|\leq \nu\quad\text{for all}\quad \gamma\in\Spec_\epsilon(p_{U_A})\text{ and }t\in B(S\cup T,\rho').\]
    \end{lemma}

    We will now record some versions of the Bogolyubov lemma, where we locate a Bohr set in many iterated sums of a set $A$ instead of just $2A-2A$. Since these results require many more iterated sums, we will call them weak Bogolyubov lemmas.

    \begin{lemma}[Global Weak Bogolyubov Lemma]\label{weakbogolyubovglobal} Let $A\subset G$. Let $\alpha = \frac{|A|}{|G|}$. Then there is a Bohr $B$ set with rank $O(\log(2/\alpha))$ and radius $\Omega(1/\log(2/\alpha))$ and a parameter $k = O(\log(2/\alpha))$ such that $B\subset kA-kA$.
    \end{lemma}
    \begin{proof} Let $\epsilon = \Omega(1)$ and $k = O(\log(2/\alpha))$ be parameters with $k$ an integer and $\epsilon^{2k-2}\leq \frac{1}{2}\alpha$. By Lemma \ref{globalchang}, we may find a Bohr set $B$ with rank $O(\log(2/\alpha))$ and radius $\Omega(1/\log(2/\alpha))$ such that $\Re(\gamma(t))\geq 0$ for all $\gamma\in\Spec_\epsilon(U_A)$ and $t\in B$. For any $t\in B$, Fourier inversion gives
    \[p_{kU_A-kU_A}(t) = \Re N^{-1}\sum_{\gamma\in\hat{G}}|\hat{p_{U_A}}(\gamma)|^{2k}\gamma(t)\geq N^{-1}|\hat{p_{U_A}}(0)|^{2k} + \Re N^{-1}\sum_{\gamma\notin \Spec_\epsilon(U_A)}|\hat{p_{U_A}}(\gamma)|^{2k}\gamma(t) \]
    \[\geq N^{-1} - \epsilon^{2k-2}N^{-1}\sum_{\gamma\in\hat{G}}|\hat{p_{U_A}}(\gamma)|^2 = N^{-1}-\epsilon^{2k-2}\frac{1}{|A|}\geq \frac{1}{2}N^{-1}>0,\]
    so $t\in kA-kA$. Since $t\in B$ was arbitrary, $B\subset kA-kA$, as desired.
    \end{proof}

    \begin{lemma}[Local Weak Bogolyubov Lemma]\label{weakbogolyubovlocal}
    Let $B = B(S,\rho)$ be a regular Bohr set of rank $d$ and let $A\subset B$. Let $\alpha = \frac{|A|}{|B|}$. There is an integer $k$ with $k = O(\log(2/\alpha))$, a set of characters $T\subset \hat{G}$ and a radius $\rho'$ with 
    \[|T| = O(\log(2/\alpha))\quad\text{ and }\rho' = \Omega(\rho\nu d^{-3}k^{-1}/\log(2/\alpha))\]
    such that $B(S\cup T,\rho')\subset kA-kA$.    
    \end{lemma}

    \begin{proof} Let $\epsilon = \Omega(1)$ and $k = O(\log(2/\alpha))$ be such that $\epsilon^{2k-2}\leq \frac{1}{32}\alpha$. Let $B' = B(S,(200kd)^{-1}\rho)$. Since $B$ is regular, there is an $x_0\in G$ such that $|A\cap (B'+x_0)|\geq (\alpha'/2)|B'|$. Set $A' = A\cap (B'+x_0)$.

    Since $B$ is regular, $|A+(k-1)A'|\leq| B(S,(1+(200d)^{-1}\rho)| \leq 2|B|$. We then estimate by Cauchy-Schwarz
    \[\|p_{U_A+(k-1)U_{A'}}\|_{\ell^2}^2 \geq 2^{-1}|B|^{-1}. \]

    We apply Lemma \ref{localchang} with $B'$ in place of $B$, $A'$ in place of $A$, to obtain a set of characters $T$ and a radius $\rho$ with 
    \[|T| = O(\log(2/\alpha))\quad\text{ and }\rho' = \Omega(\rho d^{-3}k^{-1}/\log(2/\alpha))\]
    such that $\Re(\gamma(t))\geq 1/2$ for all $\gamma\in\Spec_{\epsilon}(U_{A'})$ and $t\in B(S\cup T,\rho')$. 

    By Fourier inversion, for any $t\in B(S\cup T,\rho')$, 
    \[p_{U_A+(k-1)U_{A'}-U_A-U_{(k-1)A'}}(t) = \Re N^{-1}\sum_{\gamma\in\hat{G}}|\hat{p_{U_A}}(\gamma)|^2|\hat{p_{U_A'}}(\gamma)|^{2k-2}\gamma(t) = \]
    \[\Re N^{-1}\sum_{\gamma\in\Spec_\epsilon(U_{A'})}|\hat{p_{U_A}}(\gamma)|^2|\hat{p_{U_A'}}(\gamma)|^{2k-2}\gamma(t) + \Re N^{-1}\sum_{\gamma\notin\Spec_\epsilon(U_{A'})}|\hat{p_{U_A}}(\gamma)|^2|\hat{p_{U_A'}}(\gamma)|^{2k-2}\gamma(t)\geq \]
    \[N^{-1}2^{-1}\sum_{\gamma\in\Spec_\epsilon(U_{A'})}|\hat{p_{U_A}}(\gamma)|^2|\hat{p_{U_A'}}(\gamma)|^{2k-2} - N^{-1}\sum_{\gamma\notin\Spec_\epsilon(U_{A'})}|\hat{p_{U_A}}(\gamma)|^2|\hat{p_{U_A'}}(\gamma)|^{2k-2}\geq \]
    \[N^{-1}2^{-1}\sum_{\gamma\in\hat{G}}|\hat{p_{U_A}}(\gamma)|^2|\hat{p_{U_A'}}(\gamma)|^{2k-2} - 2N^{-1}\sum_{\gamma\notin\Spec_\epsilon(U_{A'})}|\hat{p_{U_A}}(\gamma)|^2|\hat{p_{U_A'}}(\gamma)|^{2k-2}\geq \]
    \[2^{-2}|B|^{-1} - 2\epsilon^{2k-2}N^{-1}\sum_{\gamma\in\hat{G}}|\hat{p_{U_A}}(\gamma)|^2 = 2^{-2}|B|^{-1}-2\epsilon^{2k-2}|A|^{-1} = 2^{-2}|B|^{-1}-2\epsilon^{2k-2}\alpha^{-1}|B|^{-1}\geq \]\[\frac{1}{32}|B|^{-1}>0.\]
    Since $t\in B(S\cup T,\rho')$ was arbitrary, $B(S\cup T,\rho')\subset A+(k-1)A'-A-(k-1)A'\subset kA-kA$, as desired.
    \end{proof}

\section{Inverse Theorem for $U^3$}\label{U3appendix}
In this appendix, we record a sketch of a proof of Theorem \ref{U3Inverse}. We follow \cite[Section 9]{GTinverse} nearly exactly, except we need to replace two applications of the Bogolyubov Lemma. Throughout, $G$ will be an abelian group of size $N$.

We begin with the following result, which can be used as a substitute for \cite[Proposition 9.1]{GTinverse}. 

\begin{proposition}\label{locallinearization} Let $H\subset G$, and suppose that $\xi:H\to G$ is a function whose graph $\Gamma = \{(h,\xi_h):h\in H\}$ obeys the estimates $K^{-1}N\leq |\Gamma|\leq |\Gamma-\Gamma|\leq KN$. Then there is a regular Bohr set $B(S,\rho)$ with rank $|S| = O(\log^{1+o(1)}(K))$ and radius $\Omega(\log^{-1-o(1)}(K))$, elements $x_0\in G$, $\xi_0\in \hat{G}$, and a locally linear function $M:B(S,\rho)\to \hat{G}$ such that 
\[\E(\mathds{1}_H(x_0+h)\mathds{1}_{\xi_{x_0+h}=2Mh+\xi_0}|h\in B) \geq \exp(-O(\log(K)^{1+o(1)})).\]
\end{proposition}

We divide the proof of Proposition \ref{locallinearization} into a few parts. The first is the application of Theorem \ref{improvedFreiman}.
\begin{proposition}\label{Freimaninput} Let $\Gamma$ be a set in an abelian group with $|\Gamma+\Gamma|\leq K|\Gamma|. $ Then there is a subset $\Gamma'\subset \Gamma$ with $\frac{|\Gamma'|}{|\Gamma|}\geq \exp(-O(\log(K)^{1+o(1)}))$ such that $|k\Gamma'|=k^{O(\log(K)^{1+o(1)})}|\Gamma'|$ for all $k\geq 1$.
\end{proposition}
\begin{proof} By Theorem \ref{improvedFreiman}, there is a convex coset progression $P$ of dimension at most $O(\log(K)^{1+o(1)})$ such that $|\Gamma\cap P|\geq |\Gamma|\exp(-O(\log(K)^{1+o(1)}))$. We take $\Gamma'=\Gamma\cap P$. 
\end{proof}

\begin{proposition}\label{graphrefinement} Let $\Gamma'\subset G\times \hat{G}$ be the graph of a function which obeys the estimates $|(4k+1)\Gamma'-(4k)\Gamma'|\leq L|\Gamma'|.$ Then there is a subset $\Gamma''\subset \Gamma'$ with $|\Gamma''|\geq \exp(-O(\log L \log k))\Gamma'|$ such that $2k\Gamma''-2k\Gamma''$ is a graph.
\end{proposition}
\begin{proof} Let $A = \{\xi\in\hat{G}:(0,\xi)\in 4k\Gamma'-4k\Gamma'\}$. Since $\Gamma'$ is a graph, $|\Gamma'+A|=|\Gamma'||A|$. Moreover, $|\Gamma'+A|\leq L|\Gamma'|$, so $|A|\leq L$. 

By \cite[Lemma 8.3]{GTinverse}, there is a set $S\subset \hat{\hat{G}}$ such that $A\cap B(S,1/4) = \{0\}$ and $|S|\leq 1+\log_2(L)$. 
Let $\Psi:\hat{G}\to (\mathbb{S}^1)^S$ be the homomorphism $((s(\xi))_{s\in S}$.
We cover $(\mathbb{S}^1)^S$ by $\exp(O(\log(L)\log(k)))$ cubes with side length $1/4k$. There exists one such cube $Q$ such that
$|\{(x,y)\in\Gamma':\Psi(y)\in Q\}|$ has size at least $|\Gamma'|\exp(-O(\log L \log k))$. Set $\Gamma'' = \{(x,y)\in\Gamma':\Psi(y)\in Q$.

Suppose $(x,y)$ and $(x,y')$ are in $2k\Gamma''-2k\Gamma''$. Then $y-y'\in A$ and also $y-y'\in B(S,1/4)$. Thus, $y-y'=0$, so $2k\Gamma''-2k\Gamma''$ is the graph of a function.
\end{proof}

We are now ready to prove Proposition \ref{locallinearization}.
\begin{proof}Let $k\geq 1$ be an integer to be chosen later. We apply Proposition \ref{Freimaninput} to find a subset $\Gamma'\subset \Gamma$ with $|\Gamma'|\geq \exp(-O(\log(K)^{1+o(1)})|\Gamma|$ with $|(4k+1)\Gamma'-4k\Gamma'| \leq k^{O(\log(K)^{1+o(1)})}$. We then apply Proposition \ref{graphrefinement} to find a subset $\Gamma''\subset \Gamma'$ with $|\Gamma''|\geq \exp(-O(\log(K)^{1+o(1)}\log(k)))|\Gamma|$ such that $2k\Gamma''-2k\Gamma'' $ is a graph. Let $A = \{x:(x,y)\in \Gamma''\}$. Then $|A| = |\Gamma''|\geq \exp(-O(\log(K)^{1+o(1)}\log k))N$. 

We compute $\log\frac{N}{|A|}  \leq C_1\log(K)^{1+o(1)}\log(k)$ for some absolute constant $C_1$. By Lemma \ref{weakbogolyubovglobal}, there exists an absolute constant $C_2$ such that if $k\geq C_2 \log\frac{N}{|A|}$, then there is a Bohr set $B$ with rank $O(\log(2N/|A|))$ and radius $\Omega(1/\log(2N/|A|))$ such that $B\subset kA-kA$. This can be arranged with $k = O(\log(K)^{1+o(1)})$. 

Let $f:A\to \hat{G}$ be the function with graph $\Gamma''$. We may extend $f$ to a function $\tilde{f}$ on $kA-kA$ by the formula
\[\tilde{f}\left(a_1+\dots+a_k-(b_1+\dots+b_k)\right) = f(a_1)+f(a_2)+\dots+f(a_k)-f(b_1)-\dots-f(b_k),\]
where $a_1,\dots,a_k$ and $b_1,\dots,b_k$ are elements of $A$. Since $k\Gamma''-k\Gamma''$ is the graph of a function, this formula is well-defined. We define a function $M:B\to \hat{G}$ such that $2M(x) = \tilde{f}(x)$ for all $x\in B.$ Since $2k\Gamma''-2k\Gamma''$ is the graph of a function, the function $M$ is locally linear on $B$.

The remainder of the proof of this proposition proceeds exactly as in \cite[Lemma 9.1]{GTinverse}.
\end{proof}

We may then follow the remainder of the arguments from \cite[Section 9]{GTinverse}. There are two differences. First, the lower bounds appearing on the right hand side of the equations (9.9) onwards will be of the form $\exp(-O(\log(\eta^{-1})^{1+o(1)}))$ instead of $\eta^{O(1)}$. In turn, one must dilate the radius of the Bohr set by factors of $\exp(-O(\log(\eta^{-1})^{1+o(1)}))$ instead of $\eta^{O(1)}$. The second difference is in the use of the local Bogolyubov lemma at the end of Lemma 9.4. This can be replaced by Lemma \ref{weakbogolyubovlocal}, which increases the error term a bit more, but it will still be of the form $\exp(-O(\log(\eta^{-1})^{1+o(1)}))$.\end{appendices}

\printbibliography
\end{document}